\documentclass[12pt]{amsart}
\usepackage{amsmath,amsthm,amsfonts,amssymb,eucal}

\renewcommand {\a}{ \alpha }

\newcommand{\e}{\epsilon}

\newcommand{\G}{\Gamma}

\newcommand{\vark}{\varkappa}

\renewcommand{\l}{\lambda}
\renewcommand{\L}{\Lambda}

\newcommand{\p}{\partial}

\newcommand{\Om}{\Omega}

\newcommand{\R}{ \mathbb R}

\newcommand {\GA}{\mathfrak A}

\newcommand {\GS}{\mathfrak S}

\newcommand {\GW}{\mathfrak W}
\newcommand {\GU}{\mathfrak U}

\newcommand {\bx}{\mathbf x}

\newcommand {\by}{\mathbf y}

\newcommand {\bn}{\mathbf n}

\newcommand {\bxi}{\boldsymbol\xi}

\newcommand{\plainC}[1]{\textup{{\textsf{C}}}^{#1}}

\newcommand{\plainL}[1]{\textup{{\textsf{L}}}^{#1}}

\DeclareMathOperator{\tr}{{tr}}

\DeclareMathOperator{\op}{{Op}}

\newtheorem{thm}{Theorem}

\newtheorem{lem}[thm]{Lemma}
\newtheorem{prop}[thm]{Proposition}

\theoremstyle{remark}
\newtheorem{rem}[thm]{Remark}

\begin{document}
\hoffset -4pc

\title
[On the Hankel-type operators]
{{On Hankel-type operators with discontinuous symbols in higher dimensions}}
\author{A.V. Sobolev}
\address{Department of Mathematics\\ University College London\\
Gower Street\\ London\\ WC1E 6BT UK}
\email{asobolev@math.ucl.ac.uk}
\keywords{Hankel operators, pseudo-differential operators with discontinuous symbols,
quasi-classical asymptotics}
\subjclass[2010]{Primary  47G30; Secondary 45M05, 47B10, 47B35}
\date{\today}

\begin{abstract}
We obtain an asymptotic formula for the counting function of the discrete spectrum for
 Hankel-type pseudo-differential operators with discontinuous symbols.
\end{abstract}

\maketitle

\section{Introduction}

Consider on $\plainL2(a, b), 0\le a< b\le\infty$ the integral operator of the form
\begin{equation}\label{hankel:eq}
(\G_{a, b}(k)  u)(x) = \int_a^b k(x+y) u(y) dy,
\end{equation}
with some function $k=k(t), t>0$.
The operator $\G_{0, \infty}(k)$ is called
\textsl{Hankel operator on $\plainL2(0, \infty)$}, see \cite{Peller}, p. 46.
For $\G_{a, b}(k)$ with $0\le a< b <\infty$ we use the term \textsl{truncated Hankel operator}.
The \textsl{symbol} $\vark = \vark(\xi), \xi\in\R,$ of the
operator $\G_{0, \infty}(k)$
is defined (non-uniquely) as a function such that its Fourier transform
$\hat\vark(t)$ coincides with $k(t)$ for all $t >0$.
We are interested in the case, when the
symbol $ \vark(\xi)$
is a bounded function with jump discontinuities, which ensures that
the operator $\G_{0, \infty}(k)$ is not compact.
The leading example of
such an operator is given by the \textsl{Carleman kernel}
$k(t) = t^{-1}, t>0$ (see \cite{Peller}, p. 54), for which one can choose
\begin{equation*}
\vark(\xi) =
\begin{cases}
-\pi i, \xi \le 0,\\
\pi i, \xi >0.
\end{cases}
\end{equation*}
The operator $\G_{a, b}(k)$ with this symbol is bounded for all $a$ and $b$,
$0\le a < b\le\infty$, and $\|\G_{a, \infty}(k)\| = \pi, a\ge 0$.
If $a = 0$ and/or $b = \infty$, then $\G_{a, b}(k)$ is not compact.
Among other results, H.S. Wilf
investigated the asymptotics of
the counting function of the discrete spectrum
of the truncated operator
$\G_{1, b}(k), k(t) = t^{-1},$ as $b \to\infty$ (see \cite{Wilf}, Corollary 1).
He proved that the number of eigenvalues of $\G_{1, b}(k)$ in the interval
$(\l, \infty)$ for any $\l\in (0, \pi)$ is asymptotically
equal to $C(\l) \log b$ as $b\to\infty$,
with some explicit constant $C(\l)$.
H. Widom in \cite{W_Hankel}, Theorem 4.3 derived a
similar asymptotic formula for the truncated \textsl{Hilbert matrix} (i.e.
matrix with the entries
$(j+k+1)^{-1}, \ \  j, k =  0, 1, 2, \dots$), as well as for
some more general Hankel matrices.
Later, a good deal of attention
became focused on the asymptotics of the
determinants of truncated Hankel (and Toeplitz) matrices, see  e.g.
\cite{Bas1}, \cite{Bas2}, \cite{Bas3} and \cite{DIK}.

The Hankel operator $\G_{0, \infty}(k)$ can be rewritten in the form
\begin{equation*}
(\G_{0, \infty}(k)u)(x) = (\tilde G(k)u)(-x),\ x >0,\ \
\end{equation*}
where
\begin{equation*}
(\tilde G(k) u)(x) = \bigl(1-\chi_{(0, \infty)}(x)\bigr)\int_{-\infty}^\infty k(y-x)
\chi_{(0,\infty)}(y)u(y) dy.
\end{equation*}
Here $\chi_{(0, \infty)}$ denotes the characteristic function of the half-axis
$(0, \infty)$. As A. Pushnitski and D. Yafaev \cite{PY} indicated to the author,
in the scattering theory context  it is natural to consider alongside $\G_{0, \infty}(k)$
the \textsl{symmetrised Hankel operator}
\begin{equation}\label{modi:eq}
\tilde H(k) = \tilde G(k) + \tilde G^*(k).
\end{equation}
In this note we study a multi-dimensional analog
of the truncated symmetrised Hankel operator
with a discontinuous symbol.
It is defined in the following way.
Let $\op_\a^l(a)$ and $\op_\a^r(a)$ be the standard ``left" and ``right"
pseudo-differential operators  with the smooth symbol $a=a(\bx, \bxi)$,
$\bx, \bxi\in\R^d, d\ge 1$, i.e.
\begin{equation*}
(\op^l_\a a) u(\bx) = \biggl(\frac{ \a}{2\pi}\biggr)^d
\int\int e^{i\a(\bx-\by)\bxi} a(\bx, \bxi) u(\by) d\bxi d\by,
\end{equation*}
\begin{equation*}
(\op^r_\a a) u(\bx) = \biggl(\frac{\a}{2\pi}\biggr)^d
\int\int e^{i\a(\bx-\by)\bxi} a(\by, \bxi) u(\by) d\bxi d\by,
\end{equation*}
for any function $u$ from the Schwartz class on $\R^d$. If
the symbol $a$ depends only on $\bxi$, then the above operators coincide with each other
and we simply write $\op_\a(a)$.
Here
and below integrals without indication of the domain are assumed to be taken over the entire
Euclidean space $\R^d$. The large constant
$\a\ge 1$ can be thought of as a truncation parameter.
The conditions imposed on the symbol $a$ in the main Theorem \ref{main:thm} below ensure that
the above operators are trace class for all $\a \ge 1$.

In order to introduce the jump discontinuities,
let $\L, \Om$ be two domains in $\R^d$, and let $\chi_{\L}(\bx)$, $\chi_{\Om}(\bxi)$
be their characteristic functions. We use the notation
\begin{equation*}
P_{\Om,\a} = \op_\a(\chi_{\Om}).
\end{equation*}
Define the operator
\begin{equation*}
T_\a(a) = T_\a(a; \L, \Om) = \chi_{\L} P_{\Om, \a}\op^l_\a(a) P_{\Om, \a}\chi_{\L},
\end{equation*}
and its off-diagonal version
\begin{equation*}
G_\a(a) = G_\a(a; \L, \Om) = (1-\chi_{\L})
P_{\Om, \a} \op_\a^l(a) P_{\Om, \a} \chi_{\L}.
\end{equation*}
The central object for us is the following Hankel-type self-adjoint operator
\begin{equation*}
H_\a(a) = H_\a(a; \L, \Om) = G_\a(a; \L, \Om) + G^*_\a(a; \L, \Om),
\end{equation*}
which is a natural multi-dimensional analogue
of the truncated symmetrised operator \eqref{modi:eq}.
Note the following elementary property of
$H_\a(a)$. Let $U$ be the unitary operator in $\plainL2(\R^d)$ defined by
\begin{equation*}
Uu = u\chi_{\L} - u(1-\chi_{\L}), u\in\plainL2(\R^d),
\end{equation*}
so that $U^* = U$. Then $U^* H_\a(a) U = - H_\a(a)$. This implies, in particular, that
the spectrum of $H_\a(a)$ is symmetric w.r.t. zero, i.e.
\begin{equation}\label{symmetry1:eq}
\dim\ker \bigl(H_\a(a)-\l\bigr) = \dim\ker \bigl(H_\a(a)+\l\bigr),\ \l\in\R.
\end{equation}

Let $g$ be a function analytic in a disk of a sufficiently large radius,
such that $g(0) = 0$.
In 1982 H. Widom in \cite{W_conj} conjectured an asymptotic formula
for the trace $\tr g(T_\a)$, $\a\to\infty$,
which was subsequently
proved in \cite{Sob}.
In order to state this result define
for any symbol $b = b(\bx, \bxi)$, any
domains $\L, \Om$ and any $\plainC1$-surfaces $S, P$, the coefficients
\begin{equation}\label{w0:eq}
\GW_0(b) = \GW_0(b; \L, \Om)
= \frac{1}{(2\pi)^d} \int_{\L}\int_{\Om} b(\bx, \bxi) d\bxi d\bx,
\end{equation}
\begin{equation}\label{w1:eq}
\GW_1(b) = \GW_1(b ; S, P) = \frac{1}{(2\pi)^{d-1}}\int_{S} \int_{P} b(\bx, \bxi)
| \bn_{S}(\bx)\cdot\bn_{P}(\bxi) |  dS_{\bxi} dS_{\bx},
\end{equation}
where $\bn_{S}(\bx)$ and $\bn_P(\bxi)$ denote the exterior unit normals to
$S$ and $P$ at the points $\bx$ and $\bxi$ respectively.
Define also
\begin{equation}\label{GA:eq}
\GA(g; b) = \frac{1}{(2\pi)^2}\int_0^1 \frac{g(bt) - t g(b)}{t(1-t)} dt.
\end{equation}
Then, as conjectured in \cite{W_conj} and proved in \cite{Sob},
\begin{equation}\label{widom:eq}
\tr g\bigl( T_\a(a)\bigr)
= \a^d \ \GW_0\bigl(g(a); \L, \Om\bigr)
+ \a^{d-1} \log\a \ \GW_1\bigl(\GA(g; a); \p\L, \p\Om\bigr) + o(\a^{d-1}\log\a),
\end{equation}
The aim of this note is to establish a similar trace formula for the operator $H_\a(a)$,
see Theorem \ref{main:thm}.

We use the standard notation
$\GS_1$ for the trace class operators, and $\GS_2$ for the Hilbert-Schmidt operators,
see e.g. \cite{BS}.
The underlying Hilbert space is assumed to be
 $\plainL2(\R^d)$.
By $C, c$ (with or without indices) we denote various positive
constants independent of $\a$,
whose precise value is of no importance.

\textbf{Acknowledgment.} The author is grateful to G. Rozenblum
for discussions, and to A. B\"ottcher and A. Pushnitski for stimulating remarks.
This work was supported in part by EPSRC grant EP/F029721/1.

\section{Main result}

\subsection{Definitions and main results}
To state the result define for any function $g=g(t), t\in\R$, such that $|g(t)|\le Ct$, the
integral
\begin{equation}\label{gu:eq}
\GU(g; b) = \frac{2}{\pi^2} \int_0^1\frac{g(b\frac{t}{2})}{t\sqrt{1-t^2}} dt.
\end{equation}
Denote by
\begin{equation*}
g_{\textup{\tiny ev}}(t) = \frac{g(t)+g(-t)}{2}
\end{equation*}
the even part of $g$. Let $\GW_1$ be as defined in \eqref{w1:eq}.
The next theorem contains the main result of the paper:

\begin{thm}\label{main:thm}
Let $\L, \Om\subset\R^d$, $d\ge 2$ be bounded
domains in $\R^d$ such that
$\L$ is $\plainC1$  and $\Om$ is $\plainC3$.
Let  $a = a(\bx, \bxi)$ be a symbol satisfying the condition
\begin{equation}\label{dersym:eq}
\max_{\substack{0\le n\le d+2\\ 0\le m\le d+2}}
\sup_{\bx, \bxi}
|\nabla_{\bx}^n \nabla_{\bxi}^m a(\bx, \bxi)|
< \infty,
\end{equation}
and having a compact support in both variables.
Let $g=g(t), t\in\R$ be a function such that
$g_{\textup{\tiny ev}}(t)  t^{-2}$ is continuous on $\R$. Then
\begin{align}\label{main:eq}
\tr  g(H_\a(a))
= \a^{d-1} \log\a\  \GW_1(\GU(g_{\textup{\tiny ev}}; |a|); \p\L, \p\Om) + o(\a^{d-1}\log\a),
\end{align}
as $\a\to\infty$.
\end{thm}

Note that the coefficient on the right-hand side of \eqref{main:eq}
is well-defined for any
function $g$ satisfying the conditions of Theorem \ref{main:thm}.
Note also that in view of \eqref{symmetry1:eq}, we have
\begin{equation}\label{symmetry2:eq}
\tr g(H_\a) = \tr g_{\textup{\tiny ev}}(H_\a).
\end{equation}

\begin{rem}
Denote by $n_\pm(\l; \a)$ with $\l >0$ the number of eigenvalues of
the operator $\pm H_\a(a)$ which are greater than $\l$.  In other words,
\[
n_\pm (\l_1; \a) = \tr \chi_I(\pm H_\a(a)),\ \ I = (\l, \|H_\a(a)\|+1).
\]
Since the interval $I$ does not contain the point $0$, this quantity is finite.
Due to \eqref{symmetry2:eq}, $n_+(\l;a) = n_-(\l; a)$.
In order to find the leading term of the asymptotics of
$n_{\pm}(\l; \a)$, $\a\to\infty$, approximate
the characteristic function $\chi_I$ from above and from below
by smooth functions $g$.
Then it follows from Theorem \ref{main:thm} that
\begin{equation}\label{mul:eq}
n_\pm(\l; \a) =  \frac{1}{2}\a^{d-1} \log\a  \
\GW_1\bigl(\GU(\chi_I; |a|)\bigr)  + o(\a^{d-1}\log\a).
\end{equation}
A straightforward calculation shows that
\begin{equation*}
\GU\bigl(\chi_I; |a(\bx, \bxi)|\bigr)
=
\begin{cases}
\dfrac{2}{\pi^2} \cosh^{-1} \dfrac{|a(\bx, \bxi)|}{2\l}, \ |a(\bx, \bxi)| > 2\l;\\[0.3cm]
0,\  | a(\bx, \bxi)| \le 2\l.
\end{cases}
\end{equation*}
The formula \eqref{mul:eq} can be viewed as a
multi-dimensional analogue of the asymptotics
derived for the Carleman kernel in \cite{Wilf}, Corollary 1.
\end{rem}

Theorem \ref{main:thm} will be derived from the
following theorem, which is simply formula \eqref{main:eq}
 for even polynomial functions $g$:

\begin{thm}\label{main_poly:thm}
Let the domains $\L, \Om\subset\R^d$, $d\ge 2$
and the symbol $a = a(\bx, \bxi)$ be as in Theorem \ref{main:thm}.
Then for $g_p(t) = t^p, p= 1, 2,  \dots, $
\begin{equation}\label{main_poly:eq}
\tr  g_{2p}(H_\a(a))
= \a^{d-1} \log\a\  \GW_1(\GU(g_{2p}; |a|); \p\L, \p\Om) + o(\a^{d-1}\log\a),
\end{equation}
as $\a\to\infty$.
\end{thm}

Once this theorem is proved, the asymptotics is extended to functions
$g$ satisfying the conditions of Theorem \ref{main:thm}
with the help of an elementary trace estimate for $g(H_\a)$.


\section{Auxiliary estimates}

The proof relies on various trace class estimates, some of which were obtained in \cite{Sob}.
In these estimates we always assume that the symbols $a$ and $b$ satisfy the condition
\eqref{dersym:eq} and that the domains $\L, \Om$ are as in Theorem \ref{main:thm}.

We begin with well known estimates for operators with smooth symbols:

\begin{prop}\label{product:prop}
Let $d\ge 1$ and $\a\ge c$.
Suppose that the symbols $a, b$ satisfy \eqref{dersym:eq}. Then
\begin{equation*}
\|\op_{\a}^l(a) \| + \|\op_{\a}^r(a)\|\le C.
\end{equation*}
If, in addition, $a$ and $b$ are compactly supported in both variables, then
\begin{gather*}
\|\op_\a^l(a) - \op_\a^r(a)\|_{\GS_1}\le C\a^{d-1},\\[0.2cm]
\| \op_\a^l(a) \op_\a^l(b) - \op_\a^l(ab)\|_{\GS_1}\le
C \a^{d-1}.
\end{gather*}
\end{prop}

The above boundedness is a classical fact, and it can be found, e.g. in \cite{Cordes},
Theorem $B_1'$, where it was established under
smoothness assumptions weaker than \eqref{dersym:eq}.
For the trace class estimates see e.g. \cite{Sob}, Lemma 3.12, Corollary 3.13.

The following estimates are for operators with discontinuous symbols.

\begin{prop}\label{sandwich:prop}
Suppose that the symbol $a$ satisfies \eqref{dersym:eq} and
has a compact support in both variables.
Assume that $\a\ge c$.
Let $\op_\a(a)$ denote any of the operators
$\op^l_\a(a)$ or $\op^r_\a(a)$.  Then
\begin{equation*}
\|[\op_\a(a), P_{\Om, \a}]\|_{\GS_1} +
\|[\op_\a(a), \chi_\L]\|_{\GS_1}\le C  \a^{d-1}.
\end{equation*}

\end{prop}

See \cite{Sob}, Lemmas 4.3 and 4.5.

The next proposition is a crucial ingredient in our proof:
Theorem \ref{main_poly:thm} is derived from the following
local version of the asymptotics \eqref{widom:eq}.

\begin{prop}\label{sobw:prop}
Suppose that the symbols $a$ and $b$ satisfy \eqref{dersym:eq},
and that $b$ has a compact supports in both variables.
Then for $g_p(t) = t^p, p= 1, 2,  \dots, $
\begin{align}\label{toeplitz:eq}
\tr \bigl(\op_\a^l(b) g_p(T_\a(a))\bigr)
= &\ \a^d \GW_0(b g_p(a); \L, \Om)\notag\\[0.2cm]
&\ + \a^{d-1} \log\a\  \GW_1(b\GA(g_p; a); \p\L, \p\Om) + o(\a^{d-1}\log\a),
\end{align}
as $\a\to\infty$.
\end{prop}

See \cite{Sob}, Theorem 11.1.

The following Hilbert-Schmidt estimate is important
when deriving Theorem \ref{main:thm} from Theorem
\ref{main_poly:thm}.

\begin{prop}\label{HS:prop} Suppose that the symbol $a$ satisfies \eqref{dersym:eq} and has
a compact support in both variables. Assume that $\a \ge 2$. Then
\begin{equation*}
\| G_\a(a)\|_{\GS_2}^2\le C\a^{d-1}\log\a.
\end{equation*}
\end{prop}

See \cite{Sob}, Lemma 6.1, and also papers \cite{G1, G2} by D. Gioev.
\section{Proof of Theorem \ref{main:thm}}

\subsection{Polynomial functions, reduction to the  operator $T_\a$}
We begin with studying the modulus of the operator $G_\a(1)$,
i.e. the operator $|G_\a(1)| = \sqrt{G_\a^*(1) G_\a(1)}$.

\begin{lem} \label{hantoto:lem} Suppose that the symbol $b$ satisfies
\eqref{dersym:eq} and has a compact support in both variables.
Then
\begin{equation*}
\tr \op(b) g_p\bigl(G_\a^*(1) G_\a(1)\bigr)
= \frac{1}{2}\a^{d-1}\log \a \ \GW_1\bigl(b\GU(g_{2p}; 1); \p\L, \p\Om\bigr)
+ o(\a^{d-1}\log\a),
\end{equation*}
as $\a\to\infty$.
\end{lem}

\begin{proof} Denote $G = G_\a(1)$ and $T=T_\a(1)$.
By inspection, $G^* G = T - T^2$, so $g_p(G^*G) = h(T)$,
with the polynomial $h(t) = g_p(t - t^2)$. Thus one can use Proposition
\ref{sobw:prop}. Note that $h(1) = 0$, so that the first asymptotic coefficient
in \eqref{toeplitz:eq} vanishes, see \eqref{w0:eq}. Let us find the second
asymptotic coefficient, calculating $\GA(h; 1)$. Using \eqref{GA:eq} and \eqref{gu:eq}
we get
\begin{equation*}
\GA(h; 1) = \frac{2}{(2\pi)^2}\int_0^{\frac{1}{2}} \frac{g_p(t-t^2)}{t(1-t)} dt
= \frac{1}{\pi^2} \int_0^1 \frac{g_p(\frac{s^2}{4})}{s\sqrt{1-s^2}} ds
= \frac{1}{2}\GU(g_{2p}; 1).
\end{equation*}
Thus, by Proposition \ref{sobw:prop},
\begin{equation*}
\tr \bigl(\op_\a^l(b) h(T)\bigr)
= \frac{1}{2}\a^{d-1}\log \a \ \GW_1\bigl(b\GU(g_{2p}; 1); \p\L, \p\Om\bigr)
+ o(\a^{d-1}\log\a),
\end{equation*}
which leads to the proclaimed formula.
\end{proof}

Now we can prove the asymptotics for the operator $H_\a(a)$:

\begin{proof}[Proof of Theorem \ref{main_poly:thm}]
For any pair of operators $B_1, B_2$ we write $B_1\sim B_2$ if
\[
\|B_1-B_2\|_{\GS_1}\le C\a^{d-1}, \a\ge 1.
\]
Denote $H = H_\a(a), G = G_\a(a)$, and note
that $g_{2p}(H) = g_p(H^2)$. Since $G^2 = 0$, $(G^*)^2 = 0$, it follows that
\[
H^2 = GG^*+ G^*G.
\]
This sum is in fact an orthogonal sum of two operators acting
on the mutually orthogonal subspaces
$\plainL2(\complement\L)$ and $\plainL2(\L)$, where $\complement\L$
denotes the complement of $\L$.
Moreover, the non-zero spectra of $G^* G$ and $G G^*$ are the same.
Thus
\begin{equation}\label{htog:eq}
\tr g_{2p}(H) = \tr g_p(H^2) = 2 \tr g_p(G^*G).
\end{equation}
Using Propositions \ref{sandwich:prop} and \ref{product:prop}, we conclude  that
\begin{align*}
G\sim &\ \op_\a^l(a) G_\a(1)\sim G_\a(1) \op_\a^l(a),\\
G^*\sim &\ \op_\a^l(\overline{a}) G_\a^*(1)\sim G_\a^*(1) \op_\a^l(\overline{a}),\\
G^*G\sim &\ \op_\a^l(|a|^2) G_\a^*(1) G_\a(1).
\end{align*}
Thus referring again to Proposition \ref{product:prop}, we obtain that
\begin{equation*}
g_p(G^*G)\sim \op_\a^l(|a|^{2p}) g_p(G_\a^*(1)G_\a(1)).
\end{equation*}
Now it follows from Lemma \ref{hantoto:lem} and formula \eqref{htog:eq}
that
\begin{align*}
\tr g_{2p}(H) = &\ 2\tr \bigl(\op_\a^l(|a|^{2p})
g_p(G_\a^*(1) G_\a(1))\bigr) + O(\a^{d-1})\\[0.2cm]
= &\
\a^{d-1}\log \a \ \GW_1\bigl(|a|^{2p}\GU(g_{2p}; 1); \p\L, \p\Om\bigr)
+ o(\a^{d-1}\log\a).
\end{align*}
Since $|a|^{2p}\GU(g_{2p}; 1) = \GU(g_{2p}; |a|)$, this implies
the asymptotics \eqref{main_poly:eq}.  The proof of Theorem \ref{main_poly:thm} is now
complete.
\end{proof}

\subsection{Proof of Theorem \ref{main:thm}}
By virtue of \eqref{symmetry2:eq} we may assume that $g = g_{\textup{\tiny ev}}$.
Since $h(t) = t^{-2} g_{\textup{\tiny ev}}(t)$ is continuous, for any $\e>0$ there
exists an even polynomial $p = p(t)$ and a continuous function $r = r(t)$ such that
\begin{equation*}
g_{\textup{\tiny ev}}(t) = p(t) + t^2 r(t),\ \
\max_{t\in[-\|H_\a\|, \|H_\a\|]} |r(t)| \le \e.
\end{equation*}
In view of \eqref{main_poly:eq},
\begin{equation*}
\tr p(H_\a(a)) = \a^{d-1} \log\a \ \GW_1(\GU(p; a)) + o(\a^{d-1}\log \a).
\end{equation*}
Moreover, by the definition \eqref{gu:eq},
\begin{equation*}
|\GW_1(\GU(p; a)) -\GW_1(\GU(g_{\textup{\tiny ev}}; a))|\le C \e.
\end{equation*}
Let us estimate the error term:
\begin{align*}
\| g_{\textup{\tiny ev}}(H_\a(a)) - p(H_\a(a))\|_{\GS_1}
\le &\ \|r(H_\a(a))\|\ \|H_\a^2(a)\|_{\GS_1} \\[0.2cm]
\le &\ 2\e \|G_\a(a)\|_{\GS_2}^2
\le C\e \a^{d-1}\log\a,
\end{align*}
where we have used Proposition \ref{HS:prop}.
Collecting the above estimates together, we obtain
\begin{equation*}
\limsup_{\a\to\infty}\frac{1}{\a^{d-1}\log\a}
|\tr g(H_\a(a)) - \a^{d-1}\log\a\ \GW_1(\GU(g_{\textup{\tiny ev}}; a))|
\le C\e.
\end{equation*}
Since $\e>0$ is arbitrary, the asymptotics \eqref{main:eq} follows.
\bibliographystyle{amsplain}
\bibliography{bibmaster}

\end{document}